\documentclass[11pt]{amsart}
\usepackage[numbers,sort&compress]{natbib}
\usepackage{amsfonts,bbm}
\usepackage{mathrsfs} 
\usepackage{appendix}
\usepackage{mathtools}
\usepackage{amsmath}
\usepackage{amssymb}
\usepackage{amsthm}
\usepackage{latexsym}
\RequirePackage{natbib}
\usepackage{tikz}
\usetikzlibrary{arrows,positioning,chains,fit,shapes,calc}
\newtheorem{theorem}{Theorem}[section]
\newtheorem{lemma}[theorem]{Lemma}
\newtheorem{corollary}[theorem]{Corollary}
\newtheorem{proposition}[theorem]{Proposition}

\newtheorem{remark}[theorem]{Remark}

\def\RR{{\mathbb R}}  
\def\PP{{\mathbb P}}
\def\NN{{\mathbb N}}      
\newcommand{\norm}[1]{\left\Vert #1 \right\Vert} 

\DeclareMathOperator{\diam}{diam}
\DeclareMathOperator{\fix}{Fix}

\DeclareMathOperator{\arcsinh}{arcsinh}

\begin{document}
 \title[Inexact Krasnosel'skii-Mann iterations]{Rates of convergence for inexact Krasnosel'skii-Mann iterations in Banach spaces\thanks{N\'ucleo Milenio Informaci\'on y Coordinaci\'on en Redes ICM/FIC RC130003}}


\author[M. Bravo]{Mario Bravo}
\address[Mario Bravo]{Departamento de Matem\'atica y Ciencia de la Computaci\'on, Universidad de Santiago de Chile, 
Alameda Libertador Bernardo O'higgins 3363, Santiago, Chile.
{\em E-mail}: {\tt mario.bravo.g@usach.cl}}
\thanks{{\em Acknowledgements.} This work was partially supported by N\'ucleo Milenio Informaci\'on y Coordinaci\'on en Redes ICM/FIC RC130003. 
Mario Bravo was partially funded by FONDECYT 11151003. Roberto Cominetti and Mat\'ias Pavez-Sign\'e gratefully acknowledge the support provided by FONDECYT 1130564 and FONDECYT 1171501.}

\author[R. Cominetti]{Roberto Cominetti}
\address[Roberto Cominetti]{Facultad de Ingenier\'ia y Ciencias, Universidad Adolfo Ib\'a\~nez, 
Diagonal Las Torres 2640, Santiago, Chile. 
{\em E-mail}: {\tt roberto.cominetti@uai.cl}}

\author[M. Pavez-Sign\'e]{Mat\'ias Pavez-Sign\'e}
\address[Mat\'ias Pavez-Sign\'e]{Departamento de Ingenier\'ia Matem\'atica, Universidad de Chile,  Beauchef 851, Santiago, Chile.
{\em E-mail}: {\tt mpavez@dim.uchile.cl}}



\date{}
\begin{abstract}
We study the convergence of an inexact version of the classical 
Krasnosel'skii-Mann iteration for computing fixed points of nonexpansive maps. Our main result establishes a new metric bound for the fixed-point 
residuals, from which we derive their rate of convergence as well as the convergence 
of the iterates towards a fixed point. The results are applied 
to three variants of the basic iteration: infeasible iterations 
with approximate projections, the Ishikawa iteration, and diagonal
Krasnosels'kii-Mann schemes. The results are also extended to continuous time in order 
to study the asymptotics
of nonautonomous evolution equations governed by nonexpansive operators.
\end{abstract}

\maketitle

\maketitle



\section{Introduction}
Let   $T:C\to C$ be a nonexpansive map defined on a closed convex domain $C$ in a Banach space  $(X,\|\cdot\|)$. 
The Krasnosel'skii-Mann iteration approximates a fixed point of $T$ by the sequential averaging process  
\begin{equation}\label{KM} \tag{\sc km}
  x_{n+1}= (1\!-\!\alpha_{n+1})\,x_n + \alpha_{n+1}\,Tx_n,
\end{equation}
where  $x_0\in C$ is an initial guess and $\alpha_n\in [0,1]$ is a given sequence of scalars.

This iteration, introduced by Krasnosel'skii \cite{kra} and Mann \cite{man},
 arises frequently in convex optimization as many algorithms can be cast 
in this framework. This is the case of the gradient  method for 
convex functions with Lipschitz gradient \cite{btp},  the proximal point method \cite{mar,rock},
as well as different decomposition methods such as the forward-backward 
splitting method  \cite{mer,pas}, the alternating direction method of multipliers ADMM \cite{gab}, the Douglas-Rachford splitting \cite{dr},
and the Peaceman--Rachford splitting \cite{pr}. For a  comprehensive survey
of these methods and their numerous applications we refer to \cite{bc}.
Note that \eqref{KM} also arises when discretizing the evolution equation 
$u'(t)+[I - T]u=0$ (see {\em e.g.} \cite{bb92,bre}), so that many results for \eqref{KM}
admit natural extensions to continuous time.

A central issue when studying the convergence of the iterates $x_n$ towards a fixed point of $T$ 
is to establish the strong convergence of the residuals $\|x_n-Tx_n\|\to 0$, a
property known as {\em asymptotic regularity} \cite{brp2,brp1}.  For an historical account
of results in this area we refer to  \cite{bb92,bb96}. In the case of a bounded domain  $C$, an explicit estimate
for the residual was conjectured in \cite{bb92} and  recently confirmed in \cite{csv14}, namely
\begin{equation}\label{bound}\|x_n-Tx_n\|\leq \dfrac{\diam(C)}{\sqrt{\pi\sum_{k=1}^n\alpha_k(1\!-\!\alpha_k)}},
\end{equation}
where the constant $1/\sqrt{\pi}$ is known to be tight (see \cite{bc16}).
This inequality implies that asymptotic regularity holds as soon as $\sum_{k=1}^\infty\alpha_k(1\!-\!\alpha_k)=\infty$.
The bound \eqref{bound} can also be used to estimate the number of iterations required to attain any prescribed accuracy, as well as to establish
the rate of convergence of the residuals. For instance, if $\alpha_n$ remains away from 0 and 1 the bound yields 
$\|x_n-Tx_n\|= O(1/\sqrt{n})$, whereas for $\alpha_n=1/n$ one gets an order $O(1/\sqrt{\ln n})$.

When the operator values $Tx$ can only be computed up to some precision, one is naturally led to consider the inexact iteration
 \begin{equation}\label{IKM} \tag{\sc ikm}
  x_{n+1}= (1\!-\!\alpha_{n+1})\,x_n + \alpha_{n+1}\,( Tx_n + e_{n+1}),
\end{equation}
where $e_{n+1}$ can be interpreted as an error in the evaluation of $Tx_n$, or as a perturbation
of the iteration. Note that \eqref{IKM} requires $x_n\in C$ so that it  
 assumes implicitly that the iterates remain in $C$.

This inexact iteration was used by Liu~\cite{Liu95} to study the equation $Sx=f$, restated as a fixed point  of 
$Tx=f+x-Sx$, with $S$ demicontinuous and strongly accretive on a uniformly smooth space.
Liu proved the strong convergence of the iterates
assuming that $Tx_n$ remains bounded, $\sum_{k\ge 1}\|e_k\|<\infty$, and $\alpha_n\to 0$ with $\sum_{k\ge 1}\alpha_k=\infty$. 
Weak convergence of \eqref{IKM}  was also established for $T$ nonexpansive, first on Hilbert spaces \cite[Combettes]{Combettes01} and then on
uniformly convex spaces \cite[Kim and Xu]{KIMXU07},
provided that $\fix(T)\not=\emptyset$ and $\sum_{k\ge 1}\alpha_k(1\!-\!\alpha_k)=\infty$ with $\sum_{k\ge 1}\alpha_k\|e_n\|<\infty$.
The rate of convergence of \eqref{IKM} was recently studied in a Hilbert setting by Liang, Fadili and Peyr\'e~\cite{lfp16}, proving
that $\|x_n-Tx_n\|=O(1/\sqrt{n})$ under the
stronger summability condition $\sum_{k\ge 1}k\|e_k\|<\infty$ and with $\alpha_n$  bounded away from 0 and 1. The proof exploits the Hilbert structure
and does not seem to carry over to general Banach spaces.

\subsection{Our contribution}
The main result in this paper is an extension of the bound \eqref{bound} which holds for the inexact iteration \eqref{IKM}
in general normed spaces. From this extended bound we draw a number of consequences on the convergence of the iterates and the rate 
of convergence of the fixed point residuals, and we derive continuous time analogs for the asymptotics of evolution equations governed by nonexpansive 
operators. 

In all what follows we denote $\epsilon_n\geq \|e_n\|$ a bound for the errors, and we let 
\begin{equation}
\tau_n=\mbox{$ \sum_{k=1}^n\alpha_k(1\!-\!\alpha_k).$}
\end{equation} 
We also consider the  function 
$\sigma:[0,\infty)\to\RR$ defined by 
\begin{equation}\label{omega}
\sigma(y)=\min\{1,{1}/{\sqrt{\pi y}}\}.
\end{equation} 
With these notations, our main result can be stated as follows.

\begin{theorem}\label{thm:main} 
Let $(x_n)_{n\in\NN}$ be a sequence generated by \eqref{IKM} and assume that 
\begin{equation}\label{h0}\tag{$\mbox{\sc h}_0$}
\mbox{there exists $\kappa\in [0,\infty)$ such that $x_n\in C$ and $\|Tx_n-x_0\|\le \kappa$ for all $n\in \NN$.}
\end{equation}
Then, for all $n\in\NN$ we have
\begin{equation}\label{bound_inexact}
\|x_n-Tx_n\| \leq {\kappa}\,{\sigma(\tau_n)} + \sum_{i=1}^{n}{2\,\alpha_i\epsilon_i}\,{\sigma(\tau_n\!-\!\tau_i)} + 2\,\epsilon_{n+1}.
	\end{equation}
In particular, if $\tau_n\to\infty$ and $\epsilon_n\to 0$ with $\sum_{k\geq 1}\alpha_k\epsilon_k<\infty$, 
then $\|x_n-Tx_n\|\to 0$.
	\end{theorem}

Clearly, in the exact case with $\epsilon_n\equiv 0$ the bound \eqref{bound_inexact} yields \eqref{bound}.
The proof of Theorem \ref{thm:main} is presented in Section \S\ref{sec:proof} and uses probabilistic arguments 
by reducing the analysis to the study of an associated Markov reward process in $\mathbb{Z}^2$. 
As a first consequence of this result, in \S\ref{sec:cvgce} we explain how the property $\|x_n-Tx_n\|\to 0$ can be used to
show that $\fix(T)\neq\emptyset$ as well as the convergence of the iterates $x_n$ towards a fixed point.

The assumption $(\mbox{\sc h}_0)$ above imposes two conditions: the iterates must remain in $C$ and the images $Tx_n$
are bounded. Some situations in which these conditions hold are discussed in \S\ref{sec:h0},
including the case when $T$ is defined on the whole space and either it has a bounded range, or 
$\sum_{k\geq 1}\alpha_k\epsilon_k<\infty$ and $\fix(T)\neq\emptyset$.
Note also that when $C$ is bounded one can take $\kappa=\diam(C)$ so that it suffices to ensure that
the iterates remain in $C$. Alternatively, in \S\ref{sec:bddom} we consider an iteration that uses approximate 
projections to deal with the case when $x_n$ falls outside $C$.

Section \S\ref{sec:rates} exploits the bound \eqref{bound_inexact} in order to establish several results 
on the rate of convergence of the residuals. Theorem \ref{thm:2} 
shows that if $\alpha_n$ remains away from 0 and 1 and $\sum_{k\geq 1}k^a \|e_k\|<\infty$ then $\|x_n-Tx_n\|= O(1/n^{b})$ with 
$b=\min\{\frac{1}{2},a\}$. This extends the main result in \cite{lfp16} that covers
only the case $a=1$ and is restricted to Hilbert spaces. On the other hand, from  known properties of the Gauss hypergeometric function $_2F_1(a,b;c;z)$,
 we obtain as a Corollary of Theorem \ref{thm:dos}   that when $\norm{e_n}= O({1}/{n^a})$ with $\alpha_n$ bounded away from 0 and 1,
the residual norm satisfies
$$\norm{x_n - Tx_n}=\left\{
\begin{array}{ll}
O(1/n^{a-1/2})&\mbox{ if $\frac{1}{2}\leq a<1$,}\\
O({\log n}/{\sqrt{n}})&\mbox{ if $a=1$, }\\
O(1/{\sqrt{n}})&\mbox{ if $a>1$.}
\end{array}\right.$$
Note that for $a\leq 1$ the assumption $\norm{e_n}= O({1}/{n^a})$ is very mild and allows 
 for nonsummable errors. 
Similar rates are obtained for vanishing stepsizes of the form $\alpha_n=1/n^c$ with $c\le 1$.

Section \S\ref{sec:4} explores three variants of the basic iteration \eqref{IKM}. In \S\ref{sec:bddom} we deal with the case in
which the iterates might fall outside $C$ by  using a suitable approximate projection 
of $x_n$ onto $C$. Then, in \S\ref{sec:ishikawa} we analyze the Ishikawa iteration
which can be seen as a special case of the inexact scheme \eqref{IKM}.
In \S\ref{sec:diagonal} we consider a diagonal version of \eqref{IKM} in which
the operator $T$ might change at each iteration. 

The final Section \S\ref{sec:continuous} presents the extension of the results to continuous time, 
establishing the rate of convergence for the nonautonomous evolution equation
\begin{equation*}\tag{E}\begin{cases}
u'(t)+(I-T)u(t)=f(t),&\\
u(0)=x_0.&
\end{cases}
\end{equation*}

 \section{Proof of Theorem \ref{thm:main}}\label{sec:proof}
 We begin by noting that  $x_n-Tx_n = (x_n-x_{n+1})/\alpha_{n+1}+e_{n+1}$ so that
 \begin{equation}\label{eq:basic}\|x_n-Tx_n\|\leq  \frac{\|x_n-x_{n+1}\|}{\alpha_{n+1}}+\|e_{n+1}\|.
 \end{equation}
 In order to bound the term $\|x_n-x_{n+1}\|$ we follow a similar approach as in \cite{bb96,csv14} by establishing a recursive bound for the differences 
$\|x_m-x_n\|$ for all $0\leq m \leq n$. In what follows we let $\alpha_0=1$ and 
for $0\leq i\leq n$ we denote
\begin{equation}\label{eq:pi}
\pi_i^n=\alpha_i\,\mbox{$\prod_{k=i+1}^n(1\!-\!\alpha_k)$}
\end{equation}
so that $\sum_{i=0}^n\pi_i^n=1$. For $i=n$ we use the standard convention  $\prod_{k=n+1}^n(1\!-\!\alpha_k)=1$.
The following is a slight variant of \cite[Proposition 2]{csv14}, which itself extends a similar
result by Baillon and Bruck \cite{bb92,bb96}.
\begin{lemma}\label{sym} Let $(x_n)_{n\in\NN}$ be defined inductively by 
$x_{n+1}=(1-\alpha_{n+1})x_{n}+\alpha_{n+1}y_{n}$ with $x_0\in X$ and $y_n\in X$.
Then, setting  $y_{-1}=x_0$,  we have  $x_n=\sum_{i=0}^n\pi_{i}^ny_{i-1}$ for all $n\ge 0$. Moreover, for $0\le m\le n$ it holds 
	\begin{equation}\label{eq:uno}
	x_m-x_n=\sum_{i=0}^m\sum_{j=m+1}^n\pi_i^m\pi_j^n(y_{i-1}-y_{j-1}).
	\end{equation} 
\end{lemma}
\begin{proof}
	The equality  $x_n=\sum_{i=0}^n\pi_{i}^ny_{i-1}$ follows by a straightforward inductive argument, while \eqref{eq:uno}
	follows from this equality  and the identities $\sum_{i=0}^n\pi_i^n=1$ and  $\pi_i^m-\pi^n_i=\sum_{j=m+1}^n\pi_i^m\pi_j^n$ for $0\le i\le m\le n$.
\end{proof}
We note that the sequence generated by \eqref{IKM} corresponds to $y_n=Tx_n+e_{n+1}$,
from which  we deduce the following recursive bound. 
\begin{corollary}\label{cor:2} Let $(x_n)_{n\in\NN}$ be generated by \eqref{IKM}. Assume $(\mbox{\sc h}_0)$ and let $\epsilon_n\geq\|e_n\|$. 
For each $n\in\NN$ define inductively $w_{m,n}$ for  $-1\le m \leq n$ by setting  $w_{-1,n}=\kappa$  
and
\begin{equation}\label{W}
w_{m,n}=\sum_{i=0}^m \sum_{j=m+1}^n \pi_i^m\pi_j^n(w_{i-1,j-1} + \epsilon_i + \epsilon_j)\quad \mbox{ for }m=0,\ldots,n. 
\end{equation}
 Then $\|x_m-x_n\|\le w_{m,n}$ for all $0\le m\le n$.
	\end{corollary}
	\begin{proof}
		The proof is by induction on $n$. 
		The base case $n=0$ being trivial, let us suppose that $\|x_i-x_j\|\le w_{i,j}$ holds for all $i,j$ with $0\le i \le j<n$.
Applying Lemma \ref{sym} with $y_n=Tx_n+e_{n+1}$  and setting by convention $Tx_{-1}=y_{-1}=x_0$ and $e_0=0$,
we get the  inequality 
	\begin{equation*}
	\|x_m-x_n\|\le \sum_{i=0}^m\sum_{j=m+1}^n\pi^m_i\pi^n_j(\|Tx_{i-1}-Tx_{j-1}\|+\|e_i\|+\|e_j\|).
	\end{equation*}
The terms with $i=0$ can be bounded as 
\[\|Tx_{-1}-Tx_{j-1}\|=\|x_0-Tx_{j-1}\|\leq \kappa = w_{-1,j-1},\]
	 while for the remaining terms the nonexpansivity of $T$ and the induction hypothesis give $\|Tx_{i-1}-Tx_{j-1}\|\leq \|x_{i-1}-x_{j-1}\|\leq w_{i-1,j-1}$.
Hence
	$$\|x_m-x_n\|\le \sum_{i=0}^m\sum_{j=m+1}^n\pi^m_i\pi^n_j(w_{i-1,j-1}+\epsilon_i+\epsilon_j)=w_{m,n}$$
		 which completes the induction step.	
	 	\end{proof}
		
\subsection{Reduction to a Markov reward process}\label{Ap1}
The main step in the proof of Theorem \ref{thm:main} 
relies on a probabilistic reinterpretation of $w_{m,n}$
in terms of a Markov reward process evolving in $\mathbb{Z}^2$.
The process is similar to the Markov chain used in \cite{csv14},
except that we add rewards to account for the presence of errors,
which requires a finer analysis.

Namely, let $m<n$ be positive integers and consider a race between a fox at position $n$ that is trying to catch a hare
	located at position $m$. At each integer $i\in\NN$ the fox jumps over a hurdle to
	reach the position $i-1$. The jump succeeds with probability $(1-\alpha_i)$ in which case the
	process repeats, otherwise the fox falls at $i\!-\!1$ where it gets a reward $\epsilon_i$. 
	The fox catches the hare if it jumps successfully down to $m$ or below.
	Otherwise, the hare gets a chance to run towards the burrow located at $-1$ by following the
	same rules and with the same rewards. The process alternates until either the fox catches the hare, or the
	hare reaches the burrow. In the latter case the hare gets an additional reward $\kappa$.

This yields a Markov chain with state space  $\mathcal{S}=\{(m,n): 0\le m<n\}\cup \{f,h\}$ with $f, h$ two absorbing states
that represent respectively the cases in which the fox or the hare win the race (see Figure \ref{proc}).
Specifically, starting from a transient state $(m,n)$ with $0\le m<n$, the process moves with probability $\pi_i^m\pi_j^n$ to 
a new state of the form $(i\!-\!1,j\!-\!1)$ with $1\leq i \leq m<j\leq n$,  and otherwise  it is absorbed in state $f$ with probability
$\sum_{j=0}^m\pi_j^n$ and in state $h$ with probability $\pi_0^m\sum_{j=m+1}^n\pi_j^n$.
\begin{center}
	\begin{figure}[ht]
		\begin{tikzpicture}[scale=0.45,every node/.style={scale=0.6}]
		\draw[step=1cm,gray,very thin] (-2,-2) grid (10,10); 
		\draw[thick,->] (0,0) -- (9,0);
		\draw[thick,->] (0,0) -- (0,9);
		\foreach \i in {0,1,...,6}{
			\pgfmathsetmacro{\Start}{\i+1}
			\foreach \j in {\Start,...,8} 
			\draw[black,fill=black] (\i,\j) circle (0.1cm);
		}
		\draw[dashed,-] (-1,-1) -- (8,8);
		\draw[black,fill=black] (7,8) circle (0.1cm);
		\node at (5,8.65) {\Large $(m,n)$};
		\node[diamond,fill=black] at (-1,5) {};
		\node at (-1,5.65) {\Large $h$};
		\node[diamond,fill=black] at (5,3) {};
		\node at (4.95,2.3) {\Large $f$};
		\end{tikzpicture}
		\caption{The state space $\mathcal S$.\label{proc}}
	\end{figure}
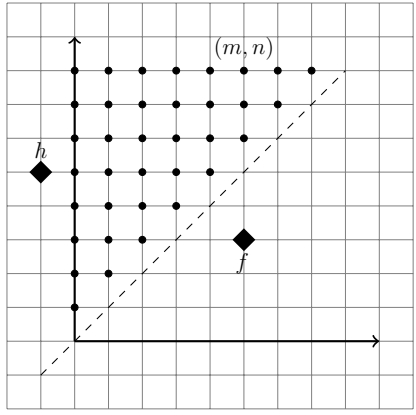
\end{center} 
\vspace{-3ex}
\noindent
When the process visits a
transient state $(i-1,j-1)$ the hare gets a reward $\epsilon_{i}$ and the fox gets a reward $\epsilon_{j}$, which combined
yield a total reward $R_{i-1,j-1}=\epsilon_{i}+\epsilon_{j}$. If the process reaches state $h$ the hare gets a 
reward $\kappa$ and the fox gets nothing, whereas at the absorbing state $f$ there is no reward. 
Then the total expected reward when the process starts at 
position $(m,n)$ with $0\le m\le n$ satisfies exactly the recursion \eqref{W} with boundary condition
$w_{-1,n}=\kappa$ corresponding to the reward collected at the absorbing state $h$.
This  provides an alternative way to compute $w_{n,n+1}$ and allows to establish the following bound.

\begin{proposition}\label{lemma1} 
Let $w_{m,n}$ be defined recursively by \eqref{W} with $w_{-1,n}=\kappa$,  where $\epsilon_n\geq 0$ and $\epsilon_0=0$. 
Then, for all $n\in\NN$ we have
\begin{equation}\label{eq:estimate1}
		\dfrac{w_{n,n+1}}{\alpha_{n+1}} \leq \kappa\, \sigma(\tau_n) + \sum_{i=1}^{n}2\alpha_i\epsilon_i\, \sigma(\tau_n\!-\!\tau_i)+\epsilon_{n+1}.
	\end{equation}
\end{proposition}
\begin{proof}

Consider the process starting at state $(n,n+1)$. Let $R^H_i$ denote the event 
in which the hare  collects the reward at the site $i-1$, so that the total expected reward 
of the hare can be expressed as
$$T^H=\kappa\,\PP(R_0^H)+\sum_{i=1}^{n}\epsilon_i\,\PP(R^H_i).$$
The event $R^H_i$ occurs iff the process visits a state $(i-1,j-1)$ for some $j>i$, that is to say,
if the hare is not captured before the $i$-th hurdle and it fails this $i$-th jump. 
 Let $F_i$ and $H_i$ denote independent Bernoulli variables
representing the failure of the jump over the $i$-th hurdle for the fox and hare respectively,
 with $\PP(F_i=1)=\PP(H_i=1)=\alpha_i$. 
Denoting $S_{i+1}$  the event that the hare is not captured before the $(i\!+\!1)$-th hurdle we have
$R_i^H=\{H_i=1\}\cap S_{i+1}$ so that $\PP(R^H_i)=\alpha_i\,\PP(S_{i+1})$.
Now, the event $S_{i+1}$ can be written as
$$S_{i+1}=\{\mbox{$\sum_{k=j}^{n+1}F_k>\sum_{k=j}^nH_k$ for all $j=i+1,\ldots,n+1$}\}$$
which translates the fact that the 
hare is not  captured provided that the fox falls more often than the hare.
For $j=n+1$ the condition above amounts to $F_{n+1}=1$ so that denoting $Z_k=F_k-H_k$ we
can write
$$S_{i+1}=\{F_{n+1}=1\}\cap\{\mbox{$\sum_{k=j}^{n}Z_k\geq 0$ for all $j=i+1,\ldots,n$}\}$$
and  we may  use  \cite[Proposition 4]{csv14} to get 
$$\PP(S_{i+1})=\alpha_{n+1}\PP(\mbox{$\sum_{k=j}^{n}Z_k\geq 0$ for all $j=i+1,\ldots,n$})\leq \alpha_{n+1}\,\sigma(\tau_n\!-\!\tau_i).$$
From this we obtain $\PP(R^H_i)\le \alpha_i\alpha_{n+1}\sigma(\tau_n\!-\!\tau_i)$ and, noting that $\alpha_0=1$, we get
\begin{equation}\label{eq:rewH}
T^H\leq \alpha_{n+1}\left[\kappa\, \sigma(\tau_n)+\sum_{i=1}^n\alpha_i\epsilon_i\,\sigma(\tau_n\!-\!\tau_i)\right].
\end{equation}
 
 A similar argument can be used to bound the total reward collected by the fox. 
 Indeed, denoting $R^F_j$ the event in which the hare  collects the reward at the site $j-1$, 
the total expected reward of the fox is
$$T^F=\sum_{j=1}^{n+1}\epsilon_j\,\PP(R^F_j).$$
In this case the event $R^F_j$ corresponds to the fact that the process visits a state $(i-1,j-1)$ for some $i\in\{1,\ldots,j-1\}$.
This requires that the fox fails the jump of the $j$-th hurdle, that the hare has not been captured, 
and that after the fox rests at $j-1$ the hare falls before reaching the burrow. Ignoring the latter condition 
we get  the inclusion
$$R^F_j\subseteq\{F_j=1\}\cap\{\mbox{$\sum_{k=i}^{n+1}F_k>\sum_{k=i}^nH_k$ for all $i=j,\ldots,n+1$}\}.$$
For $j=n+1$ this gives $R^F_{n+1}\subseteq\{F_{n+1}=1\}$ so that $\PP(R^F_{n+1})\leq\alpha_{n+1}$.
Now,  for $j=1,\ldots,n$ the condition $\sum_{k=i}^{n+1}F_k>\sum_{k=i}^nH_k$  is superflous when $i=j$ as it follows
 from  the same condition for $i=j+1$ and the fact that $F_j=1$. Hence we have
$R^F_j\subseteq\{F_j=1\}\cap S_{j+1}$ which yields as before the upper bound $\PP(R^F_j)\le\alpha_j\alpha_{n+1}\,\sigma(\tau_n\!-\!\tau_j)$. 
From these bounds we get
\begin{equation}\label{eq:rewF}
T^F\leq\alpha_{n+1}\left[\sum_{j=1}^n\alpha_j\epsilon_j\,\sigma(\tau_n\!-\!\tau_j)+\epsilon_{n+1}\right].
\end{equation}
Since $w_{n,n+1}=T^H+T^F$, combining \eqref{eq:rewH} and \eqref{eq:rewF} we readily get
\eqref{eq:estimate1}.
\end{proof}

\subsection{Proof of Theorem \ref{thm:main}}
Using   \eqref{eq:basic} and Corollary \ref{cor:2} we get
\begin{equation*}\|x_n-Tx_n\|\leq  \frac{w_{n,n+1}}{\alpha_{n+1}}+\epsilon_{n+1}
 \end{equation*}
which combined with \eqref{eq:estimate1} yields  \eqref{bound_inexact}. 
It remains to show that $\|x_n-Tx_n\|\to 0$ when $\tau_n\to\infty$, $\epsilon_n\to 0$, and $\sum_{k\geq 1}\alpha_k\epsilon_k<\infty$. Using \eqref{bound_inexact} 
and considering a fixed $m\in\NN$, we can use the bound $\sigma(\tau_n\!-\tau_i)\leq 1$ for the terms $i=m+1,\ldots,n$ to get
\begin{eqnarray}\label{eq:rate1}
\|x_n-Tx_n\|&\leq&
{\kappa}\,{\sigma(\tau_n)}
+\sum_{i=1}^{m}{2\alpha_i\epsilon_i}\,{\sigma(\tau_n\!-\!\tau_i)}
+\sum_{i=m+1}^{n}2\alpha_i\epsilon_i
+2\epsilon_{n+1}.
\end{eqnarray}
Since $\sigma(\tau_n\!-\tau_i)\to 0$ as $n\to\infty$ we obtain
$\limsup_{n\to\infty}\|x_n-Tx_n\|
\leq \sum_{i=m+1}^{\infty}2\alpha_i\epsilon_i$
so that the conclusion follows by letting $m\to\infty$.
\hfill$\Box$

\subsection{Convergence of the iterates}\label{sec:cvgce}
Under some additional conditions, the fact that $\|x_n-Tx_n\|\to 0$ implies 
the existence of fixed points and the convergence of the \eqref{IKM} iteration.
The arguments are quite standard (see {\em e.g.}  \cite[Goebel and Kirk]{gok}) 
but for completeness we sketch the proof. 
We recall that $X$ is said to have Opial's property if for every weakly convergent sequence 
$x_n\rightharpoonup x$ we have
$$\liminf_{n\to\infty}\|x_n-x\|<\liminf_{n\to\infty}\|x_n-y\|\quad\forall y\neq x.$$

 \begin{theorem}\label{thm:cgce} Let $(x_n)_{n\in\NN}$ be generated by \eqref{IKM} and suppose that 
    $\|x_n\!-\!Tx_n\|\to 0$ and $\sum_{k\geq 1}\alpha_k\|e_k\|\!<\!\infty$.
\\[1ex]
{\em a)} If $T(C)$ is relatively compact then $x_n$ converges strongly to a fixed point of $T$.\\[1ex]
{\em b)} If $X$ is uniformly convex and $x_n$ remains bounded then $\fix(T)\neq\emptyset$. 
Moreover, if $X$ satisfies Opial's property then $x_n$ converges weakly to a fixed point of $T$. 
  \end{theorem}
 \begin{proof} From \eqref{mmm} we see that for  $x\in\fix(T)$ the sequence $\|x_{n}-x\|+\sum_{k>n}\alpha_k\|e_k\|$  
 decreases with $n$ and hence it converges. Since the tail $\sum_{k> n}\alpha_k\|e_k\|$ tends to 0 it 
 follows that  the limit $\ell(x)=\lim_{n\to\infty}\|x_n-x\|$ is well defined.

 Now, in case a) we may extract a strongly convergent subsequence $T(x_{n_k})\to  x$.
 Since $\|x_n-Tx_n\|\to 0$ we also have $x_{n_k}\to  x$ and therefore $x\in\fix(T)$.
 It follows that $\lim_{n\to\infty}\|x_{n}-x\|=\lim_{k\to\infty}\|x_{n_k}-x\|=0$ so that $x_n\to x$ in the strong sense.
 
 In case b) we can extract a weakly convergent 
 subsequence $x_{n_k}\rightharpoonup x$ and
 since $I-T$ is demiclosed (see \cite{bro67}) the assumption $\|x_n-Tx_n\|\to 0$ 
implies that $x$ is a fixed point, hence $\fix(T)\neq\emptyset$. Moreover,
Opial's property implies that $x_n$ has only one weak cluster point: if $x_{n'_k}\rightharpoonup y$ is another 
weakly convergent subsequence with $y\neq x$ then
\begin{eqnarray*}
&&\ell(x)=\liminf_{k\to\infty}\|x_{n_k}-x\|<\liminf_{k\to\infty}\|x_{n_k}-y\|=\ell(y)\\
&&\ell(y)=\liminf_{k\to\infty}\|x_{n'_k}-y\|<\liminf_{k\to\infty}\|x_{n'_k}-x\|=\ell(x)
\end{eqnarray*}
which yields a contradiction and therefore $x_n\rightharpoonup x$.
 \end{proof}
\begin{remark} \label{Rk1} {\em 
1) The existence of fixed points when $T(C)$ is relatively compact goes back to 
the original work of Krasnosel'skii in 1955, whereas for $C$ bounded in a uniformly convex space 
this was proved in 1965 independently by Browder, G\"ohde, and Kirk.\\[0.5ex]
2) Without the summability condition $\sum_{k\geq 1}\alpha_k\|e_k\|<\infty$ the iterates might fail 
to converge as illustrated by the trivial example $Tx=x$ where $x_n=\sum_{k=1}^n\alpha_k e_k$.}
\end{remark}

 \subsection{The assumption $(\mbox{\sc h}_0)$}\label{sec:h0}
 Theorem \ref{thm:main} is based on assumption $(\mbox{\sc h}_0)$ which requires 
 simultaneously that $i$) the iterates remain in $C$, and $ii$) $\|Tx_n-x_0\|\leq\kappa$ for some constant $\kappa$.
 
 For a bounded domain $C$ property $ii$) holds with $\kappa=\diam(C)$ so that one only needs to check $i$).
 This holds automatically for the exact iteration \eqref{KM} and more generally when the errors $e_{n+1}$
 are such that $Tx_n+e_{n+1}\in C$.
Note also that if $X$ is a Hilbert space one could replace   
 $T$ by $\tilde T=T\circ P_C$ where $P_C$ is the projection onto the closed convex set $C$ so that $\tilde T:X\to C$ 
 and $(\mbox{\sc h}_0)$  holds with $\kappa=\diam(C)$. When $X$ is not a Hilbert space the projection $P_C$ 
 might not exist and, even if it does, it might fail to be nonexpansive. Moreover, even in a Hilbert 
 setting the exact projection $P_C$ might be difficult to compute exactly. To deal with these cases, in \S\ref{sec:bddom}  
 we will show how \eqref{IKM} can be adapted using approximate projections. 

When $T$ is defined on the whole space condition $i$) is trivial and one
only has to check $ii$). The following result describes two simple situations where this holds.
\begin{proposition}\label{p123} Let $T:X\to X$ be a nonexpansive map.\\[1ex]
{\em a)} If $T$ has a bounded range then \eqref{h0} holds with $\kappa=\sup_{x\in X}\|Tx-x_0\|$.\\[1ex]
{\em b)} If $\fix(T)\neq\emptyset$ and the sum $S=\sum_{k=1}^\infty\alpha_k\|e_k\|$ is finite
then \eqref{h0} holds with $\kappa=2\,\mbox{\rm dist}(x_0,\fix(T))+S.$
\end{proposition}
\begin{proof}
Property a) is self evident. In order to prove b) we note that for any given   
$x\in\fix(T)$ we have
\begin{eqnarray}\label{mmm}
 \norm{x_{n}-x}&=&\norm{(1-\alpha_{n})x_{n-1}+\alpha_{n}(Tx_{n-1}+e_{n})-x}\nonumber\\
	&\le&(1-\alpha_{n})\|x_{n-1}-x\|+\alpha_{n}\|Tx_{n-1}-Tx\|+\alpha_{n}\|e_{n}\|\\
	&\le&\|x_{n-1}-x\|+\alpha_{n}\|e_{n}\|\nonumber
	\end{eqnarray}
from which we get inductively 
\begin{equation}\label{eq:est}
 \|x_n-x\|\leq \|x_0-x\|+\mbox{$\sum_{k=1}^n\epsilon_k$}\|e_k\|\leq \|x_0-x\|+S
\end{equation}
and then 
$$\|Tx_n-x_0\|\leq \|Tx_n-x\|+\|x-x_0\|\leq  \|x_n-x\|+\|x-x_0\|\leq 2\|x_0-x\|+S.$$
The conclusion follows by taking the infimum over $x\in\fix(T)$.
\end{proof}
\vspace{1ex}
\begin{remark}{\em 
From \eqref{eq:est} we get $x_n\in B(x_0,\kappa)$ so  in part b) of the previous Proposition it suffices $T$\\ to be defined on a domain $C$ that contains this ball.}
\end{remark}

\section{Rates of convergence}
\label{sec:rates}
In this section we use the bound \eqref{bound_inexact} to estimate the rate of convergence of the fixed point residuals.

\begin{theorem}\label{thm:2}
Assume $(\mbox{\sc h}_0)$. Suppose  that $\sum_{k\geq 1}\|e_k\|<\infty$ and 
that  $\alpha_n$ is bounded away from 0 and 1. Then there exists a constant $\nu\ge 0$ such that
	\begin{equation}\label{bound:2}
	\|x_n-Tx_n\|\le \frac{\nu}{\sqrt{n}}+\mbox{$\sum_{i\ge \lfloor\frac{n}{2}\rfloor}2\|e_i\|$}.
	\end{equation}
	Moreover, if  $\varphi:[0,\infty)\to[0,\infty)$ is nondecreasing and
$\mu=\sum_{k\geq 1}\varphi(k)\|e_k\|<\infty$, then  
\begin{equation}\label{bound:3}
	\|x_n-Tx_n\|\le \frac{\nu}{\sqrt{n}}+\frac{2\mu}{\varphi(\lfloor\frac{n}{2}\rfloor)}.
	\end{equation}
In particular, if
$\sum_{k\geq 1}k^a \|e_k\|<\infty$ for some $a\ge 0$ then
  $\|x_n-Tx_n\|= O(1/n^b)$ with $b=\min\{\frac{1}{2},a\}$.
\end{theorem}
\begin{proof} Take $\beta>0$ such that $\alpha_n(1\!-\!\alpha_n)\geq\beta$ for all $n\geq 1$ and define  $\nu=(\kappa+2\sqrt{2}\,\sum_{k\geq 1}\|e_k\|)/\sqrt{\pi\beta}$.
Since $\sigma(\tau_n\!-\!\tau_i)\leq 1/\sqrt{\pi\beta(n-i)}$ and $\alpha_i\leq 1$, the inequality \eqref{bound:2} follows directly from 
\eqref{eq:rate1} by
taking $m=\lfloor\frac{n}{2}\rfloor$, while \eqref{bound:3}
follows from this and the inequality 
\begin{equation}\label{ser}
\varphi(m)\sum_{k\ge m}\|e_k\|\leq \sum_{k\ge m}\varphi(k)\|e_k\|\leq \mu.
\end{equation}
The last claim $\|x_n-Tx_n\|= O(1/n^b)$ follows from \eqref{bound:3}  by taking $\varphi(k)=k^a$.
\end{proof}
\begin{remark} {\em 
Note that in \eqref{ser} the tail $\mu_m=\sum_{k\ge m}\varphi(k)\|e_k\|$ tends to 0
so that  $\sum_{k\ge m}\|e_k\|=o(1/\varphi(m))$. In particular, for $a<1/2$ the last claim in the previous result 
can be strengthened to $\|x_n-Tx_n\|= o(1/n^a)$.}
\end{remark}

The previous result derives a rate of convergence from a control on the sum $\sum_{k\geq 1}\varphi(k)\|e_k\|<\infty$. 
The next Lemma deals with the case where we control the errors $\|e_n\|\leq\epsilon_n$ rather than their sum.
\begin{lemma}\label{lemma111}
Let  $\eta=\sqrt{1+{4}/{\pi}}$. If $\epsilon_n \leq (1\!-\!\alpha_n)f(\tau_n)$ with $f : [0,\infty) \to [0,\infty)$ nonincreasing, then
\begin{equation}\label{term}
\sum_{i=1}^{n}{2\alpha_i\epsilon_i}\,{\sigma(\tau_n\!-\!\tau_i)} \leq\eta \int_{0}^{\tau_{n}} \!\!\!\frac{f(s)}{\sqrt{\tau_n-s}}\;ds
\end{equation}
\end{lemma}

\begin{proof} 
For $s\in[\tau_{i-1},\tau_{i}]$ we have $\tau_n \!-\! s\leq \tau_n\! - \!\tau_{i-1}=\tau_n \!-\! \tau_i+\alpha_i(1\!-\!\alpha_i)$. 
Since $\tau_n\! -\! \tau_i\leq \frac{1}{\pi}\,\sigma(\tau_n\!-\!\tau_i)^{-2}$ and  $\alpha_i(1\!-\!\alpha_i)\leq\frac{1}{4}\leq \frac{1}{4}\,\sigma(\tau_n\!-\!\tau_i)^{-2}$
it follows that
$\tau_n - s\leq (\mbox{$\frac{1}{\pi}+\frac{1}{4}$})\,\sigma(\tau_n\! -\! \tau_i)^{-2}$. This, combined with
the monotonicity of $f(\cdot)$, yields
$2f(\tau_i)\sigma(\tau_n\!-\!\tau_i)\leq \eta\, f(s)/\sqrt{\tau_n\!-\! s}$ so that \eqref{term} follows by integrating over the interval $[\tau_{i-1},\tau_i]$ and then summing for $i=1,\ldots,n$. 
\end{proof}

\begin{theorem}\label{thm:dos}
Assume $(\mbox{\sc h}_0)$. Suppose that $\tau_n\to\infty$ and $\norm{e_n}= O((1\!-\!\alpha_n)/\tau_n^a)$. \\[0.5ex]
{\em a)} If $\frac{1}{2}\le a<1$ then $\norm{x_n - Tx_n}= O(1/\tau_n^{a-1/2})$.\\[0.5ex]
{\em b)} If $a=1$ then $\norm{x_n - Tx_n}= O(\log \tau_n/\sqrt{\tau_n})$.\\[0.5ex]
{\em c)} If $a>1$ then $\norm{x_n - Tx_n}= O(1/\sqrt{\tau_n})$.
\end{theorem}

\begin{proof} 
Let us consider the three terms in the bound \eqref{bound_inexact}. The first term is of order $\kappa\,\sigma(\tau_n)= O(1/\sqrt{\tau_n})$ while the third term
is $2\|e_{n+1}\|= O(1/\tau_{n+1}^a)$ so that for $a\geq 1/2$ it is also $O(1/\sqrt{\tau_n})$. 
To estimate the sum in the middle term we note that $\|e_n\|\leq (1-\alpha_n) f(\tau_n)$ with $f(s)=K/(s+1)^a$ for some constant 
$K\geq 0$. Hence, denoting $I_a(t)=\int_0^t\frac{1}{(s+1)^a\sqrt{t-s}}\,ds$ and using Lemma~\ref{lemma111} we get
$$\sum_{i=1}^{n}{2\alpha_i \norm{e_i}}\,{\sigma(\tau_n\!-\!\tau_i)} \leq \eta\, K\, I_a(\tau_n).$$
For $a=1$ we have $I_1(t)=\frac{2\arcsinh\sqrt{t}}{\sqrt{t+1}}= O({\log t}/{\sqrt{t}})$ which
yields {\em b)}. For $a\neq 1$ we may use the change of variables $s=t\,(1\!-\!x)$ to express $I_a(t)$ using the  Gauss hypergeometric function $_2F_1(a,b;c;z)$, namely
$$I_a(t)=\mbox{$\frac{\sqrt{t}}{(t+1)^a}$}\int_0^1\frac{1}{(1-\frac{t}{t+1} x)^a\sqrt{x}}\,dx=\mbox{$\frac{2\sqrt{t}}{(t+1)^a} \;_2F_{1}(a,\frac{1}{2};\frac{3}{2};\frac{t}{t+1})$}.$$
Now, for $z\sim 0$ we have $_2F_1(a,b;c;z)\sim 1$, while setting $d=a+b-c$ we have the following identity (see \cite[page 559, equation 15.3.6]{abs}) which is valid when
 $d$ is not an integer and $|\arg(z)|<\pi$
$$_2F_1(a,b;c;1\!-\!z)=z^{-d}\mbox{$\frac{\Gamma(c)\Gamma(d)}{\Gamma(a)\Gamma(b)} \,\! _2F_1(c\!-\!a,c\!-\!b;1\!-\!d;z)$}
+\mbox{$ \frac{\Gamma(c)\Gamma(-d)}{\Gamma(c-a)\Gamma(c-b)} \, _2F_1(a,b;d\!+\!1;z)$.}
$$
Taking $z=\frac{1}{t+1}$  in this identity with $b=\frac{1}{2}$ and $c=\frac{3}{2}$, it follows that for $t$ large 
$$I_a(t)\sim \mbox{$\frac{1}{a-1}\,\frac{1}{\sqrt{t}}$}+\mbox{$\frac{\sqrt{\pi}\,\Gamma(1-a)}{\Gamma(\frac{3}{2}-a)}\;\frac{1}{t^{a-1/2}}$.}$$ 
From this we deduce both {\em a)} and {\em c)}, except when $a\geq 2$ is an integer since in this case 
 $\Gamma(1\!-\!a)$ has a pole. However for $a\geq 2$ the rate $\norm{e_n}= O((1\!-\!\alpha_n)/\tau_n^a)$ is 
 stronger than the same condition with $a\in (1,2)$ so that we still get the conclusion
 $\norm{x_n - Tx_n}= O(1/\sqrt{\tau_n})$.
\end{proof}

\begin{corollary} Assume \eqref{h0}. Suppose that  $\alpha_n$ is bounded away from $0$ and $1$, and $\norm{e_n}= O(1/n^a)$.
\\[1ex]
{\em a)} If $\frac{1}{2}\leq a<1$ then $\norm{x_n - Tx_n}= O(1/n^{a-1/2})$.\\[1ex]
{\em b)} If $a=1$ then $\norm{x_n - Tx_n}= O(\log n/\sqrt{n})$.\\[1ex]
{\em c)} If $a>1$ then $\norm{x_n - Tx_n}= O(1/\sqrt{n})$.
\end{corollary}
\begin{proof} Since $\alpha_n$ is far from 0 and 1 we have $\tau_n= O(n)$ and the result follows directly from Theorem \ref{thm:dos}. 
\end{proof}

\begin{remark}{\em
 For $\frac{1}{2}\leq a\leq 1$ the condition  $\norm{e_n}= O(1/n^a)$  is 
very mild and allows for nonsummable errors. 
However, this only implies a rate for $\|x_n-Tx_n\|$ and  not  the convergence of the iterates
which  in general  requires the errors to be summable (see Theorem \ref{thm:cgce} and Remark \ref{Rk1}).}
\end{remark}
Theorem \ref{thm:dos}  also gives rates of convergence for vanishing stepsizes of the form $\alpha_n=1/n^c$ with $c\leq 1$. 
We record the case $\alpha_n=1/n$ which is often used.

\begin{corollary} Assume \eqref{h0} and $\alpha_n=1/n$, and suppose that $\norm{e_n}= O(1/\log^a n)$.
\\[1ex]
{\em a)} If $\frac{1}{2}\leq a<1$ then $\norm{x_n - Tx_n}= O(1/\log^{a-1/2}n)$.\\[1ex]
{\em b)} If $a=1$ then $\norm{x_n - Tx_n}= O(\log \log n/\sqrt{\log n})$.\\[1ex]
{\em c)} If $a>1$ then $\norm{x_n - Tx_n}= O(1/\sqrt{\log n})$.
\end{corollary}

\section{Variants of the \eqref{IKM} iteration}\label{sec:4}

\subsection{Inexact projections}\label{sec:bddom}
Up to now we assumed $(\mbox{\sc h}_0)$ which requires that the iterates
$x_n$ remain in $C$. This is a nontrivial assumption that has to be checked independently.
Alternatively one might use the metric projection $P_C:X\to C$, namely $P_C(x)=\arg \min_{z\in C}\|x-z\|$, 
and consider the  iteration 
\begin{equation}
\label{KP} \tag{IKM$_p$}x_{n+1}=(1\!-\!\alpha_{n+1})x_{n}+\alpha_{n+1}(T\!\circ\! P_Cx_{n}+e_{n+1}).\end{equation}
As noted in \S\ref{sec:h0}, if $P_C$ is well defined and nonexpansive, which is the case when $X$ is a Hilbert space, the results
in the previous sections apply directly by considering the map $T\!\circ\! P_C$ instead of $T$.
However, in more general spaces the projection might not exist and even if it exists it might fail to be nonexpansive. 
On the other hand, even in a Hilbert setting the projection might be hard to compute. To overcome these difficulties 
one may consider to perform an inexact projection by choosing a sequence $\gamma_n\geq 0$ and  starting from $x_0\in C$ 
iterate as follows
\begin{equation}
\label{KMZ}\tag{IKM$_z$}
\left \{ \begin{aligned}
&\mbox{take $z_n\in C$ with $\|z_n-x_n\|\le d(x_n,C)+ \gamma_n$ and} \\
&\mbox{set  } x_{n+1}=(1\!-\!\alpha_{n+1})x_{n}+\alpha_{n+1}(Tz_{n}+e_{n+1}). 
\end{aligned} \right.
\end{equation}
In general, finding $z_n\in C$ as above requires a specific algorithm.
Simple cases where this can be done are when $C$ is a ball or the positive orthant in an $L^p$ space with $1\leq p\leq\infty$. In these cases 
the projection might fail to be nonexpansive and might even be nonunique.
\begin{theorem}\label{thm:main2} Let the sequence $(x_n,z_n)$ be given by \eqref{KMZ} with $\|e_n\|\!\to\! 0$ and
$\sum_{k\geq 1}(\alpha_k\|e_k\|\!+\!\gamma_k)\!<\!\infty$. Suppose that 
$\sum_{k\geq 1}\alpha_k(1\!-\!\alpha_k)\!=\!\infty$ and $\|x_0-Tz_n\|\leq\kappa$ for some $\kappa\geq 0$.
Then 
$\|x_n-z_n\|\to 0$ and $\|z_n-Tz_n\|\to 0$.
\end{theorem}
\begin{proof}  
Let us denote $\delta_n=d(x_n,C)+ \gamma_n$. Lemma \ref{tech} in Appendix \ref{Ap2} shows that $\delta_n$ tends to 0 so that $\|x_n-z_n\|\leq\delta_n\to 0$. On the other hand
\begin{equation}\label{bb1}
\|z_n-Tz_n\|\leq \|z_n-x_n\|+\|x_n-Tz_n\|
\leq  \delta_n + \frac{\|x_{n+1}-x_n\|}{\alpha_{n+1}}+\|e_{n+1}\|
\end{equation}   
so that it remains to show that $\frac{\|x_{n+1}-x_n\|}{\alpha_{n+1}}$ tends to 0. We proceed as
before by establishing a recursive bound $\|x_m-x_n\|\leq w_{m,n}$.
Taking $y_n=Tz_n+e_{n+1}$ with $y_{-1}=x_0$ and using Lemma~\ref{sym}  we get
$$	\|x_m-x_n\|\le \sum_{i=0}^m\sum_{j=m+1}^n\pi_i^m\pi_j^m\|y_{i-1}-y_{j-1}\|.$$
Set $w_{-1,n}=\kappa$ for all $n\in\NN$ and denote $\epsilon_n=\|e_{n}\|+\delta_{n-1}$ with $\epsilon_0=0$.
The terms with $i=0$ in the previous sum can be bounded as
$$\|y_{-1}-y_{j-1}\|\leq \|x_0-Tz_{j-1}\|+\|e_j\|\leq w_{-1,j-1}+\epsilon_0+\epsilon_j.$$
On the other hand, since $\|z_i-z_j\|\le \|x_i-x_j\| +\delta_i+\delta_j$, the nonexpansivity of $T$ implies that for $j>i\geq 1$  
$$\|y_{i-1}-y_{j-1}\|\le \|Tz_{i-1}-Tz_{j-1}\| +\|e_i\|+\|e_j\|\leq \|x_{i-1}-x_{j-1}\| +\epsilon_i+\epsilon_j.$$
Proceeding as in Corollary \ref{cor:2} we get $\|x_m-x_n\|\leq w_{m,n}$ with $w_{m,n}$ defined recursively by \eqref{W},
and then Proposition \ref{lemma1} yields
\begin{equation}\label{bb2}
\frac{\|x_{n+1}-x_n\|}{\alpha_{n+1}}\leq\frac{w_{n,n+1}}{\alpha_{n+1}}\leq {\kappa}\,{\sigma(\tau_n)}+\sum_{i=1}^n{2\alpha_i\epsilon_i}\,{\sigma(\tau_n\!-\!\tau_i)}+\epsilon_{n+1}.
\end{equation}
Since Lemma \ref{tech} shows that $\epsilon_n\to 0$ and $\sum_{k\geq 1}\alpha_k\epsilon_k<\infty$, by
arguing as in the proof of Theorem \ref{thm:main} we deduce that $\frac{\|x_{n+1}-x_n\|}{\alpha_{n+1}}$ converges to $0$ 
as claimed. \end{proof}

\begin{corollary} Under the same conditions of Theorem {\em \ref{thm:main2}} the following holds.
\\[1ex]
{\em a)} If $T(C)$ is relatively compact then $x_n$ converges strongly to a fixed point of $T$.\\[1ex]
{\em b)} If $x_n$ remains bounded  and $X$ is uniformly convex with Opial's property, then $x_n$ converges weakly to a fixed point of $T$. 

\end{corollary}
\begin{proof} 
Let $\epsilon_n=\|e_{n}\|+\delta_{n-1}$ as in the previous proof so that $\sum_{k\geq 1}\alpha_k\epsilon_k<\infty$. 
For each $x\in\fix(T)$ a simple computation yields $\|x_{n}-x\|\le\|x_{n-1}-x\|+\alpha_{n}\epsilon_n$
so that the sequence $\|x_{n}-x\|+\sum_{k>n}\alpha_k\epsilon_k$ decreases with $n$, and then 
$\|x_{n}-x\|$ converges. Since $\|z_n-x_n\|\to 0$ it follows that
$\|z_{n}-x\|$ converges as well. Then, since $\|z_n-Tz_n\|\to 0$, we may argue as in the proof of Theorem \ref{thm:cgce}
to get the strong/weak convergence of $z_n$, and hence the corresponding convergence of $x_n$.
\end{proof}

\begin{remark} {\em  The bound for $\delta_n$ in Lemma \ref{tech}, together with \eqref{bb1} and \eqref{bb2},
provide an explicit estimate for $\|z_n-Tz_n\|$ from which one can study its rate of convergence using similar 
techniques as in Section \S\ref{sec:rates}}.
\end{remark}

\subsection{Ishikawa iteration}\label{sec:ishikawa}
In \cite{Ishikawa74} Ishikawa proposed an alternative method to approximate a fixed point of 
a nonexpansive $T:C\to C$. Namely, given two
sequences  $\alpha_n,\beta_n\in (0,1)$ and starting from $x_0\in C$, the Ishikawa process generates a sequence 
by the following two-stage iteration
\begin{equation}\label{IK}\tag{I}
\left\{\begin{aligned}y_{n}&=(1\!-\!\beta_{n+1})x_{n}+\beta_{n+1}Tx_{n}\\
x_{n+1}&=(1\!-\!\alpha_{n+1})x_{n}+\alpha_{n+1}Ty_{n}\end{aligned}\right.
\end{equation}
In this subsection we assume that $C$ is bounded and we denote $\kappa=\diam(C)$.
\begin{corollary}\label{thm:ishi}Let the sequence $(x_n)$ be given by the iteration \eqref{IK} with $\beta_n\to 0$ and $\sum_{k\geq 1}\alpha_k\beta_k<\infty$,
and assume  that $\sum_{k\geq 1}\alpha_k(1\!-\!\alpha_k)=\infty$. Then
$\|x_n-Tx_n\| \to 0$ and the following estimate holds
	 \begin{equation}\label{ishi1}
		\|x_n-Tx_n\|\le \kappa\left[{\sigma(\tau_n)}+\sum_{i=1}^n {\alpha_i\beta_i}\,{\sigma(\tau_n\!-\!\tau_i)}+2 \beta_{n+1}\right].
		\end{equation}
\end{corollary}
\begin{proof} 
We observe that \eqref{IK} can be written as an \eqref{IKM} iteration with errors given by $e_{n+1}= Ty_n-Tx_n$. 
Since $Ty_n\in C$ the iterates $x_n$ remain in $C$ while by nonexpansivity we have 
$$\|e_{n+1}\|\le \|y_n-x_n\|=\beta_{n+1}\|x_n-Tx_n\|\le \kappa\, \beta_{n+1}$$
so the result follows directly from Theorem~\ref{thm:main}.
\end{proof}

\begin{remark} {\em
Ishikawa proved in \cite{Ishikawa74} that if $C$ is a convex compact subset of a Hilbert space $X$,  
the iteration \eqref{IK} converges strongly to a fixed point as soon as $0\!\le\!\alpha_n\!\leq\!\beta_n\!\leq\! 1$ 
with $\beta_n\to 0$ and $\sum_{k\geq 1}\alpha_k\beta_k=\infty$. 
Interestingly,  Corollary \ref{thm:ishi} together with Theorem \ref{thm:cgce} implies  the 
convergence when $\sum_{k\geq 1}\alpha_k\beta_k<\infty$ which is complementary to Ishikawa's condition.
Note also that we do not require $\alpha_n\le\beta_n$.
On the other hand, Ishikawa's theorem holds for the larger class of Lipschitzian pseudo-contractive 
maps, whereas our result is restricted to nonexpansive maps but 
is valid in more general spaces and it yields the rate of convergence of the fixed-point residual as in 
\S\ref{sec:rates}.}
\end{remark}

\subsection{Diagonal KM iteration}\label{sec:diagonal}
Let $T_n:C\to C$ be a sequence of nonexpansive maps converging
uniformly to $T$ so that $\rho_n=\sup\{\|T_nx-Tx\|:x\in C\}$ tends to 0.
Starting from  $x_0\in C$ consider the diagonal iteration 
\begin{equation}
\label{KMu}\tag{\sc dkm}x_{n+1}=(1-\alpha_{n+1})x_{n}+\alpha_{n+1}T_{n+1}x_{n}.
\end{equation} 

\begin{corollary}\label{diag:1}Let $x_n$ be a sequence generated by \eqref{KMu} with $\rho_n\to 0$ and 
$\sum_{k\geq 1}\alpha_k\rho_{k}<\infty$. Suppose that 
$\sum_{k\geq 1}\alpha_k(1\!-\!\alpha_k)\!=\!\infty$ and $\|x_0-Tx_n\|\leq\kappa$ for some $\kappa\geq 0$.
Then $\|x_n-Tx_n\|\to 0$ and the following estimate holds
 \[
\norm{x_n - Tx_n}\leq {\kappa}\,{\sigma(\tau_n)}+ \sum_{i=1}^{n} {2\alpha_i \rho_i}\,{\sigma(\tau_n\!-\!\tau_i)} + 2\rho_{n+1}.
 \]
\end{corollary}
\begin{proof} 
Note that \eqref{KMu} corresponds to an \eqref{IKM} iteration with errors given by $e_{n+1}=T_{n+1}x_{n}-Tx_{n}$. 
Since $T_nx_n\in C$ the iterates remain in $C$, and $\|e_n\|\leq\rho_n$.
Hence the result follows again from Theorem~\ref{thm:main}.
\end{proof}

\begin{remark} {\em
The diagonal iteration \eqref{KMu} was introduced in~\cite[Zhao and Yang]{Zhao} in order to compute a solution for the split feasibility problem in Hilbert spaces. Weak convergence of \eqref{KMu} was established in \cite[Xu]{XU06}  for uniformly convex spaces with a differentiable norm, under the same assumptions of Corollary~\ref{diag:1}. Our result shows that this also holds for uniformly convex spaces with Opial's property, and moreover it yields rates of convergence for the residuals  in the same way as in \S\ref{sec:rates}.}
\end{remark}

\section{Application to nonautonomous evolution equations}\label{sec:continuous}

Let $T:X\to X$ be a nonexpansive map and $f:[0,\infty)\to X$ a  continuous function. 
Let $u:[0,\infty)\to X$ be the unique solution of the evolution equation
\begin{equation}\label{E}\tag{E}\begin{cases}
		u'(t)+(I-T)u(t)=f(t),&\\
		u(0)=x_0.&\\\end{cases}
\end{equation}
In the autonomous case with $f(t)\equiv 0$, Baillon and Bruck \cite{bb96}
 used the Krasnosel'skii-Mann iteration to prove that $\|u'(t)\|=O(1/\sqrt{t})$,
 assuming that $T:C\to C$ with $C$ a bounded closed convex domain. 
In the nonautonomous case $u(t)$ could leave the domain $C$
so we assume that $T$ is defined on the whole space.

In order to deal with the unboundedness of the domain, and inspired from Proposition \ref{p123}, 
we consider a continuous scalar function $\epsilon(t)\geq \|f(t)\|$ and we assume one of the following alternative 
conditions 
$$
\begin{array}{ll}
(\mbox{\sc h}_2')&\mbox{\em $T$ has a bounded range, and then we let $\kappa=\sup\{\|Tx-x_0\|:x\in X\}$, or}\\[1ex]
(\mbox{\sc h}''_2)& \mbox{\em $\fix(T)\neq\emptyset$ and $\epsilon(t)$ is decreasing with $S=\int_0^\infty \epsilon(t)\,dt<\infty$ ,
and  we let } 
 \mbox{\em $\kappa=2\mathop{\rm dist}(x_0,\fix(T))+S$.}
\end{array}
$$
Under either one of these conditions we have the following analog of Theorem \ref{thm:main}.

\begin{theorem} Let $u(t)$ be the solution of \eqref{E} and assume $(\mbox{\sc h}_2')$ or $(\mbox{\sc h}_2'')$.  Then 
	\begin{equation}\label{contrate}\|u'(t)\|\le {\kappa}\,{\sigma(t)}+\int_0^t{2\,\epsilon(s)}{\sigma(t\!-\!s)}\,ds+\epsilon(t).
	\end{equation}
Moreover, if $\epsilon(t)\to 0$ and $\int_0^\infty \epsilon(s)\,ds<\infty$ then $\|u'(t)\|\to 0$ as $t\to\infty$.
	\end{theorem}
\begin{proof}Fix $t>0$ and set $\lambda^n=\frac{t}{n}$. Let us consider the sequence $(x_k^n)_{k\geq 0}$ defined by $x_0^{n}=x_0$ and \begin{equation*}\frac{x_{k+1}^n-x_k^n}{\lambda^n}=-(I-T)x_k^n+f((k\!+\!1)\lambda^n).\end{equation*}
It is well known that $x_n^n\to u(t)$ and $(x_{n+1}^n-x_{n}^n)/\lambda^n\to u'(t)$ as $n\to\infty$. 
On the other hand, $x_k^n$ corresponds to the $k$-th term of an \eqref{IKM} iteration with errors $e_k^n=f(k\lambda^n)$ and constant stepsizes $\alpha_k\equiv\lambda^n$.
We claim that \eqref{h0} holds with $\kappa$ defined as in $(\mbox{\sc h}_2')$ or $(\mbox{\sc h}_2'')$.
Indeed, in the case $(\mbox{\sc h}_2')$ this follows directly from Proposition \ref{p123}\,a),
whereas in the case $(\mbox{\sc h}_2'')$ it follows from Proposition \ref{p123}\,b) and the estimate
$$\sum_{k=1}^\infty\lambda^n\|e_k^n\|\leq \sum_{k=1}^\infty\lambda^n\epsilon(k\lambda^n)\leq\int_0^\infty\!\!\!\epsilon(s)\,ds=S.$$
Hence, letting $\tau^n_k=\sum_{i=1}^k\alpha_i(1\!-\!\alpha_i)$ and invoking Proposition~\ref{lemma1} we get
\begin{equation}\label{riemann}
\dfrac{\|x_{n+1}^n-x_{n}^n\|}{\lambda^n}\le \kappa \,\sigma(\tau^n_{n})+2\displaystyle\sum_{i=1}^{n}\mbox{$\frac{t}{n}$}\,\epsilon(i\mbox{$\frac{t}{n}$})\,\sigma(\tau^n_n\!-\!\tau^n_i)+\epsilon((n\!+\!1)\mbox{$\frac{t}{n}$}).
\end{equation}
Since $\tau_n^n=t(1\!-\!\frac{t}{n})\to t$ as $n\to\infty$ the first term
$\kappa\,\sigma(\tau^n_{n})$ converges to $\kappa\,\sigma(t)$,
while for the third term we have $\epsilon((n\!+\!1)\mbox{$\frac{t}{n}$})\to\epsilon(t)$.
Also $\tau_n^n\!-\!\tau_i^n=(1\!-\!\frac{t}{n})(t\!-\!i\frac{t}{n})$
so that the middle term is a Riemann sum for the function
$h_n(s)=2\epsilon(s)\sigma((1\!-\!\frac{t}{n})(t\!-\!s))$. Since
$h_n(s)$ converges uniformly for $s\in [0,t]$ towards $h(s)=2\epsilon(s)\sigma(t\!-\!s)$,
this Riemann sum converges as $n\to\infty$ to the integral $\int_0^t 2\epsilon(s)\sigma(t\!-\!s)ds$. Therefore
 by letting $n\to\infty$ in \eqref{riemann} we obtain  \eqref{contrate}.

To prove the last claim $\|u'(t)\|\to 0$ we note that $\sigma(t)\to 0$ for $t\to\infty$ while $\epsilon(t)\to 0$ by assumption, 
so that it suffices to prove that $\int_0^t 2\epsilon(s)\sigma(t\!-\!s)ds$ tends to 0 as $t\to\infty$.
Denoting $h_t(s)=2\epsilon(s)\sigma(t\!-\!s)\mathbbm{1}_{[0,t]}(s)$ this integral is exactly $\int_{\RR}h_t(s)ds$.
Now, the definition of $\sigma(\cdot)$ implies $h_t(s)\to 0$ pointwise as $t\to\infty$,
and since $h_t(s)\leq 2\epsilon(s)$, the conclusion follows from Lebesgue's dominated 
convergence theorem.
\end{proof}

Clearly, from \eqref{contrate} we also get
\begin{equation*}\|u(t)-Tu(t)\|\le {\kappa}\,{\sigma(t)}+\int_0^t{2\,\epsilon(s)}{\sigma(t\!-\!s)}\,ds+2\epsilon(t)
	\end{equation*}
so that  $\|u(t)-Tu(t)\|\to 0$ as soon as $\epsilon(t)\to 0$ and $\int_0^\infty \epsilon(s)\,ds<\infty$. 
As in the discrete setting, from this one can deduce that $\fix(T)\neq\emptyset$
as well as the convergence of $u(t)$ to a fixed point of $T$. 
 \begin{theorem} Let $u(t)$ be the solution of  \eqref{E}. Suppose that $\int_{0}^\infty \epsilon(s)\, ds<\infty$
 and   $\|u(t)-Tu(t)\|\to 0$.
\\[1ex]
{\em a)} If $T(C)$ is relatively compact then $u(t)$ converges strongly to a fixed point of $T$.\\[1ex]
{\em b)} If $X$ is uniformly convex and $u(t)$ remains bounded then $\fix(T)\neq\emptyset$. 
Moreover, if $X$ satisfies Opial's property then $u(t)$ converges weakly to a fixed point of $T$. 
  \end{theorem}

 \begin{proof} We claim that for all $x\!\in\!\fix(T)$ the limit $\ell(x)=\lim_{t\to\infty} \|u(t)-x\|$ exists.
 To prove this let $\theta(t)=\frac{1}{2}\|u(t)-x\|^2$ and $g(t)\!=\!\sqrt{2\theta(t)+1}+\int_{t}^\infty \!\epsilon(s)\,ds$.
 In order to establish the existence of the limit $\ell(x)$ it suffices to show that $g(t)$
is decreasing. Let us prove that $g'(t)\leq 0$, that is to say, $\frac{d}{dt}\sqrt{2\theta(t)+1}\leq \epsilon(t)$. 
We recall that the duality mapping on 
$X$  is  the subdifferential $J(x)=\partial \psi (x)$ of the convex function $\psi(x)=\frac{1}{2}\|\cdot\|^2$. 
Choosing $u^*(t)\in J(u(t)\!-\!x)$, the subdifferential inequality gives
$$\mbox{$\frac{1}{2}\|u(t)-x\|^2 +\langle u^*(t),v-u(t)\rangle\leq \frac{1}{2}\|v-x\|^2$}\qquad(\forall v\in X)$$
so that taking $v=u(t-h)$ with $h>0$ we get
$$\langle u^*(t),u(t-h)-u(t)\rangle\leq \theta(t-h)-\theta(t).$$
Dividing by $h$ and letting $h\downarrow 0$ it follows that $\theta'(t)\leq \langle u^*(t),u'(t)\rangle$. Then, using the equation \eqref{E}
and the fact that $x\in\fix(T)$, the nonexpansivity of $T$ gives
\begin{eqnarray*}
\theta'(t)&\leq& \langle u^*(t),(I-T)x-(I-T)u(t)+f(t)\rangle\\
&=&\langle u^*(t),Tu(t)-Tx\rangle-\langle u^*(t),u(t)-x\rangle+\langle u^*(t),f(t)\rangle\\
&\leq&\|u^*(t)\| \|u(t)-x\| -\langle u^*(t),u(t)-x\rangle+\|u^*(t)\| \|f(t)\|.
\end{eqnarray*}
Now, from well known properties of the duality mapping we have
$\langle u^*(t),u(t)-x\rangle=\|u(t)-x\|^2$ and $\|u^*(t)\|=\|u(t)-x\|$ so that 
$$\theta'(t)\leq \|u^*(t)\| \|f(t)\| =  \|u(t)-x\| \|f(t)\| \leq \sqrt{2\theta(t)+1}\; \epsilon(t)$$
which proves our claim $\frac{d}{dt}\sqrt{2\theta(t)+1}\leq \epsilon(t)$. This implies the
existence of $\ell(x)=\lim_{t\to\infty} \|u(t)-x\|$, from which the rest of the proof follows  the same pattern as  the proof of Theorem \ref{thm:cgce}. 
 \end{proof}
\vspace{2ex}
The estimate \eqref{contrate} can also be used to derive the following continuous time analogs of the
rates of convergence in Theorem \ref{thm:2} and Theorem \ref{thm:dos}.

\begin{theorem}
Let $u(t)$ be the unique solution of \eqref{E}. Assume $(\mbox{\sc h}_2'')$ and let
$\nu=(\kappa+2\sqrt{2}\,S)/\sqrt{\pi}$. Then, for all $t\geq 1$ we have
	\begin{equation}\label{contrate1}\|u'(t)\|\le \dfrac{\nu}{\sqrt{t}}+\int_{t/2}^\infty{4\,\epsilon(s)}\,ds.
	\end{equation}
Moreover, if  $\varphi:[0,\infty)\to[0,\infty)$ is nondecreasing and
$\mu=\int_0^\infty\varphi(s)\epsilon(s)\,ds<\infty$, then  
\begin{equation}\label{bound:3c}
	\|u'(t)\|\le \dfrac{\nu}{\sqrt{t}}+\frac{4\mu}{\varphi(t/2)}.
	\end{equation}
In particular, if
$\int_0^\infty s^a\, \epsilon(s)\,ds<\infty$ for some $a\geq 0$, then
  $\|u'(t)\|= O(1/t^{b})$ with $b=\min\{\frac{1}{2},a\}$.
\end{theorem}
\begin{proof} Let us fix $t\geq 1$ and consider the bound \eqref{contrate}. Splitting
the integral $\int_0^t{2\epsilon(s)}{\sigma(t\!-\!s)}\,ds$ into $[0,\frac{t}{2}]$ and $[\frac{t}{2},t]$,
and noting that $\sigma(t-s)\leq \sqrt{{2/\pi t}}$ on the first interval and $\sigma(t-s)\leq 1$ on the second, we get
\begin{equation}\label{trtr}
\|u'(t)\|\le\frac{\nu}{\sqrt{\pi t}}+\int_{t/2}^t{2\,\epsilon(s)}\,ds+\epsilon(t).
\end{equation}
Since $\epsilon(\cdot)$ is nonincreasing and $t\geq 1$, we have  
$\epsilon(t)\leq\int_{t/2}^t2\,\epsilon(s)\,ds$
which plugged into \eqref{trtr} yields \eqref{contrate1}.	
Now, since $\varphi(\cdot)$ is nondecreasing, \eqref{bound:3c} follows directly from \eqref{contrate1} using the inequality 
$$\varphi(\mbox{$\frac{t}{2}$})\int_{t/2}^\infty\epsilon(s)\,ds \leq \int_{t/2}^\infty\varphi(s)\epsilon(s)\,ds\leq \mu,$$
while the last claim $\|u'(t)\|= O(1/t^{b})$ follows from \eqref{bound:3c}  by taking $\varphi(s)=s^a$.
\end{proof}
  
\begin{theorem}
Let $u(t)$ be the solution of \eqref{E} and assume $(\mbox{\sc h}_2')$ or $(\mbox{\sc h}_2'')$, and
$\epsilon(t)= O(1/t^a)$ with $a\geq \frac{1}{2}$.\\[0.5ex]
{\em a)} If $\frac{1}{2}\leq a<1$ then $\|u'(t)\|= O(1/t^{a-1/2})$.\\[0.5ex]
{\em b)} If $a=1$ then $\|u'(t)\|= O(\log t/\sqrt{t})$.\\[0.5ex]
{\em c)} If $a>1$ then $\|u'(t)\|= O(1/\sqrt{t})$.
\end{theorem}

\begin{proof} This follows from \eqref{contrate} and from the analysis of the asymptotics of the 
integral $I_a(t)=\int_0^t\!\frac{1}{(s+1)^a\sqrt{t-s}}\,ds$ established in the proof of Theorem \ref{thm:dos}. 
\end{proof} 

\appendix

\section{Bound for approximate projections}\label{Ap2}
The goal of this Appendix is to establish the next technical Lemma used in the proof of Theorem \ref{thm:main2}.
\begin{lemma}\label{tech} Let $(x_n,z_n)$ be given by \eqref{KMZ}, and denote $\delta_n=d(x_n,C)+ \gamma_n$ 
and $\xi_n=\|e_n\|+\gamma_n/\alpha_n$ with $\delta_0=\xi_0=0$.
If $\sum_{k\geq 1}\alpha_k=\infty$ and $\sum_{k\geq 1}\alpha_k\xi_k<\infty$, 
then $\delta_n\to 0$ and $\sum_{k\geq 1}\alpha_{k}\delta_{k-1}<\infty$. 
\end{lemma}
\begin{proof} 
Starting from the identity  
$$x_{n}=(1-\alpha_{n})z_{n-1}+\alpha_{n}Tz_{n-1}+(1-\alpha_{n})(x_{n-1}-z_{n-1})+\alpha_{n}e_{n}$$
and since $(1-\alpha_{n})z_{n-1}+\alpha_{n}Tz_{n-1}\in C$, we get
$$\mbox{dist}(x_{n},C)\leq \|(1-\alpha_{n})(x_{n-1}-z_{n-1})+\alpha_{n}e_{n}\|\le (1-\alpha_{n})\delta_{n-1}+\alpha_{n}\|e_{n}\|.$$
It follows that $\delta_{n}\leq (1-\alpha_{n})\delta_{n-1}+\alpha_{n}\xi_n$
so that letting $\rho_n=\prod_{j=1}^n(1\!-\!\alpha_j)$ with $\rho_0=1$ we get
$$\frac{\delta_{n}}{\rho_n}\leq \frac{\delta_{n-1}}{\rho_{n-1}}+\frac{\alpha_{n}}{\rho_n}\xi_n.$$
Iterating this inequality we get
$\frac{\delta_{n}}{\rho_n}\leq \sum_{i=0}^n\frac{\alpha_{i}}{\rho_i}\xi_i$  which yields
$\delta_{n}\le \sum_{i=0}^n\alpha_i\xi_i\,\frac{\rho_n}{\rho_i}.$

This inequality can be written as $\delta_{n}\le \int_{\NN}f_n\,d\mu$ with
$\mu$ the finite measure on $\NN$ defined by $\mu(\{i\})=\alpha_i\xi_i$, and 
$f_n:\NN\to\RR$ given by 
$f_n(i)=\frac{\rho_n}{\rho_i}$ for $i\le n$ and $f_n(i)=0$ for $i>n$.
Since $f_n(i)\to 0$ as $n\to\infty$ and $f_n(i)\leq 1$, Lebesgue's dominated convergence
theorem implies that $\int_{\NN}f_nd\mu$ tends to zero so that $\delta_n\to 0$.

It remains to show that the sum $S=\sum_{k\geq 1}\alpha_{k}\delta_{k-1}$ is finite. 
Using the previous bound for $\delta_{k-1}$ and exchanging the order of  summation
we get
$$S\leq \sum_{k=1}^\infty\alpha_{k}\sum_{i=0}^{k-1}\alpha_i\xi_i\,\frac{\rho_{k-1}}{\rho_i}
=\sum_{i=0}^\infty\alpha_i\xi_i\sum_{k=i+1}^\infty \mbox{$\alpha_{k}\prod_{j=i+1}^{k-1}(1\!-\!\alpha_j)$}.
$$
The term $q_{i+1}^k=\alpha_{k}\prod_{j=i+1}^{k-1}(1\!-\!\alpha_j)$ in this last sum can be interpreted as a probability.
Namely, suppose that at every integer $j$ we toss a coin that falls head with probability $\alpha_j$.
Then, $q_{i+1}^k$ is the probability that starting at position $i+1$ the first head occurs exactly at position $k$.
Hence $\sum_{k=i+1}^\infty q_{i+1}^k=1$ and therefore
$S\leq\sum_{i=0}^\infty\alpha_i\xi_i<\infty$.
\end{proof}

\bibliographystyle{ims}
\bibliography{NE2}


\end{document}